\documentclass[11pt]{amsproc}

\usepackage{amsfonts}

\newtheorem{definition}{Definition}

\newtheorem{lemma}{Lemma}

\newtheorem{proposition}{Proposition}

\newtheorem{theorem}{Theorem}
\numberwithin{equation}{section}
\begin{document}
\title[The Local Time of the Classical Risk Process]{The Local Time of the
Classical Risk Process${^*}$}
\author{F. Cortes }
\address{Universidad Aut\'{o}noma de Aguascalientes, Departamento de
Matem\'aticas y F\'{\i}sica, Av. Universidad 940, C.P. 20100 Aguascalientes,
Ags., Mexico}
\email{fcortes@correo.uaa.mx}
\author{J.A. Le\'{o}n}
\address{Cinvestav-IPN, Departamento de Control Autom\'{a}tico, Apartado
Postal 14-740, 07000 M\'exico D.F., Mexico}
\email{jleon@ctrl.cinvestav.mx}
\thanks{$^*$Partially supported by the CONACyT grant 45684-F, and by the UAA
grants PIM 05-3 and PIM 08-2}
\author{J. Villa}
\address{Universidad Aut\'{o}noma de Aguascalientes, Departamento de
Matem\'aticas y F\'{\i}sica, Av. Universidad 940, C.P. 20100 Aguascalientes,
Ags., Mexico}
\email{jvilla@correo.uaa.mx}
\date{December 26, 1997}

\begin{abstract}
In this paper we give an explicit expression for the local time of the
classical risk process and associate it with the density of an occupational
measure. To do so, we approximate the local time by a suitable sequence of
absolutely continuous random fields. Also, as an application, we analyze the
mean of the times $s \in [0,T]$ such that $0\leq X_{s} \leq
X_{s+\varepsilon} $ for some given $\varepsilon>0$.
\end{abstract}

\maketitle

\section{Introduction and main results}

Henceforth, $X=\{X_{t},t\geq 0\}$ represents the classical risk process.
More precisely, 
\begin{equation*}
X_{t}=x_{0}+ct-\sum_{k=1}^{N_{t}}R_{k},\ \ t\geq 0,
\end{equation*}
where $x_{0}\geq 0$ is the initial capital, $c>0$ is the premium income per
unit of time, $N=\left\{ N_{t},t\geq 0\right\}$ is an homogeneous Poisson
process with rate $\alpha $ and $\left\{ R_{k},k\in\mathbb{N}\right\}$ is a
sequence of i.i.d non-negative random variables, which is independent of $N$%
. $N_{t}$ is interpreted as the number of claims arrivals during time $t$\
and $R_{k}$ as the amount of the $k$-th claim. We suppose that $R_{1}$ has
finite mean and it is an absolutely continuous random variable with respect
to the Lebesgue measure.

The risk process has been studied extensively because it is often used to
describe the capital of an insurance company. Indeed, among the properties
of $X$ considered by several authors, we can metion that the local time of $X
$ has been analyzed by Kolkovska et al. \cite{T-J-J}, the double Laplace
transform of an occupation measure of $X$ has been obtained by Chiu and Yin 
\cite{C-Y}, or that the probability of ruin has been one of the most
important goals of the risk theory (see, for example, Asmussen \cite{A},
Grandell \cite{Grandell}, Rolski et al. \cite{R-S-S-T} and the references
therein to get an idea of the analysis realized in this subject). In this
paper we are interested in continuing the development of the local time $L$
of $X$ and its applications as an occupational density in order to improve
the understanding of $X$.

Note that $X$ is a L\'{e}vy process due to $\sum_{k=1}^{N}R_{k}$ being a
compound Poisson process. Thus, we can apply different criteria for general L%
\'{e}vy processes to guarantee the existence of $L$. For example, we can use
the Hawkes' result \cite{Hawkes} when $R_{k}$ is exponential distributed
(see also \cite{Bertoin} and references therein for related works). However,
we cannot obtain in general the form of $L$ via this results. Moreover, in
the literature there exist different characterizations of the local time
(see Fitzsimmons and Port \cite{F-P} and the references therein). For
instance, the local time have been introduced in \cite{F-P} (resp. \cite%
{T-J-J}) as an $L^{2}(\Omega )$-derivative (resp. derivative in probability)
of some occupation measure. Nevertheless, in \cite{F-P,T-J-J}, it is not
analyzed some properties of the involved local time using this
\textquotedblleft approximation of $L$\textquotedblright .

The purpose of this paper is to associate the local time of $X$ with the
crossing process when $L$ is interpreted as a density of the occupational
measure (see Theorem 1.c) below). The relation between the local time and
the crossing process was conjectured by L\'evy \cite{L} for the Brownian
motion case (i.e., whe $X$ is a Wiener process). In this article we use the
ideas of the proof of Tanaka's formula for the Brownian motion (see Chung 
\cite{Chung}, Chapter 7) to obtain a sequence of absolutely continuous
random fields (in time) that converges with probability 1 (w.p.1 for short)
to 
\begin{eqnarray}
L_{t}(x) &=&\frac{1}{c}(\frac{1}{2}1_{\{x\}}(X_{t})+1_{(x,\infty )}(X_{t})-%
\frac{1}{2}1_{\{x\}}(x_{0})-1_{(x,\infty )}(x_{0})  \notag \\
&&-\sum_{0<s\leq t}\{1_{(x,\infty )}(X_{s})-1_{(x,\infty )}(X_{s-})\}),\quad
t\ge 0 \text{\ and\ } x\in\mathbb{R}.  \label{tlrp1}
\end{eqnarray}
This approximation allows us to prove that this $L$ is the density of the
occupation measure (see (\ref{ocumea}) below) and, therefore, to deal with
some problems related to occupations measures.

Notice that $L$ given by (\ref{tlrp1}) is well-defined because $X$ is
c\`adl\`ag and 
\begin{equation}
P(N_{t}<+\infty ,\, \text{for all\ } t>0)=1,  \label{fsptp}
\end{equation}
wich imply that only a finite number of summands in (\ref{tlrp1}) are
different than zero.

In the following result we not only relate $L$ to the number of crossings
with certain level, but also to the occupation measure 
\begin{equation}
Y_{t}(A)=\int_{0}^{t}1_{A}(X_{s})ds,\quad t\geq 0\ \text{and}\ A\in \mathcal{%
B}(\mathbb{R}),  \label{ocumea}
\end{equation}%
where $\mathcal{B}(\mathbb{R})$ is the Borel $\sigma$-algrebra of $\mathbb{R}
$. Toward this end, we need the following:

\begin{definition}
\label{defcruza} We say that there exists a \textrm{crossing with the level $%
x\in \mathbb{R}$ at time $s\in (0,+\infty )$} if for all open interval $I$
such that $s\in I$, $x$ is an interior point of $\left\{ X_{t}:t\in
I\right\} $. That is, $x\in \left( X\left( I\right) \right) ^{\circ }.$
Moreover, the number of crossings with the level $x$ in the interval $(0,t)$
is denoted by $C_{t}(x)$. $C$ is known as the \textrm{crossing process of} $%
X $.
\end{definition}

\noindent Observe that if $x\in\mathbb{R}$ is a crossing point at time $s$,
then $X$ is continuous at time $s$ and $X_s=x$.

Now we can state the main result of the paper.

\begin{theorem}
\label{theop} Let $t>0$ and $x\in\mathbb{R}$. Then, the random field $L$
defined in (\ref{tlrp1}) has the following properties:

\begin{itemize}
\item[a)] $L_{t}(x)\geq 0$ and $L_{\cdot }(x)$ is not decreasing w.p.1.

\item[b)] $L_{t}(x)=\frac{1}{c}\left( \frac{1}{2}1_{\{X_{t}\}}(x)-\frac{1}{2}%
1_{\left\{ X_{0}\right\} }(x)+C_{t}(x)\right) $ w.p.1.

\item[c)] For every bounded and Borel measurable function $g:\mathbb{R}%
\mathbb{\rightarrow }\mathbb{R}$, we have 
\begin{equation}
\int_{0}^{t}g(X_{s})ds=\int_{\mathbb{R}}g(y)L_{t}(y)dy\quad w.p.1.
\label{dtl}
\end{equation}
\end{itemize}
\end{theorem}

Note that Statement b) implies that the number of crossings $C$ of $X$
introduced in Definition \ref{defcruza} satisfies 
\begin{equation*}
C_{t}(x)=1_{(-\infty ,X_{t})}(x)-1_{(-\infty ,X_{0})}(x)+\sum_{0<s\leq
t}1_{(X_{s},X_{s-})}(x)\quad w.p.1,
\end{equation*}%
for $t>0$ and $x\in\mathbb{R}$. Also note that, from (\ref{dtl}) and
Statement a), the random field $L$ can be interpreted as an occupation
density relative to the Lebesgue measure on $\mathbb{R}$. Hence, $L$ in (\ref%
{tlrp1}) is called \textit{the local time} and the expression%
\begin{eqnarray*}
L_{t}(x) &=&\frac{1}{c}(\frac{1}{2}1_{\{X_{t}\}}(x)-\frac{1}{2}1_{\left\{
X_{0}\right\} }(x)+1_{(-\infty ,X_{t})}(x)-1_{(-\infty ,X_{0})}(x) \\
&&+\int_{(0,t]}f(x,X_{s})dX_{s}),
\end{eqnarray*}%
is known as \textit{Tanaka-like formula for }$L_{t}(x)$. Here 
\begin{equation*}
f(x,X_{s})=\left\{ 
\begin{tabular}{ll}
$\frac{1_{(X_{s},X_{s-})}(x)}{\Delta X_{s}},$ & $\Delta X_{s}\neq 0,$ \\ 
$0,$ & $\Delta X_{s}=0.$%
\end{tabular}%
\right.
\end{equation*}

On the other hand, relation (\ref{dtl}) can be extended to some occupational
results. Indeed, as an example, we can state the following, which leads us
to get some average of the pathwise behavior of $X$.

\begin{theorem}
\label{teomed2}Let $g:\mathbb{R}\times \mathbb{R}\longrightarrow \mathbb{R}$
be a bounded and Borel measurable function. Then for each $\varepsilon >0,$%
\begin{equation}
E[\int_{0}^{t}g(X_{s},X_{s+\varepsilon }-X_{s})ds]=\int_{\mathbb{R}%
}E[g(x,X_{\varepsilon }-x_{0})]E[L_{t}(x)]dx.  \label{cm2}
\end{equation}
\end{theorem}

An application of this theorem is to answer the question: \textit{What is
the average in time that the capital of an insurance company is positive,
and bigger than itself after twelve months?.}

The paper is organized as follows. In Section \ref{sec:2} we provide the
tool needed to prove Theorem \ref{theop}. In particular, we approximate the
local time by a sequence of suitable random fields. The proof of Theorem \ref%
{theop} is given in Section \ref{sec:3}. Finally, in Section \ref{sec:4}, we
show Theorem \ref{teomed2} and answer the above question in the case that
the claim $R_{1}$ has exponential distribution.

\section{Main tool}

\label{sec:2} In this section we provide the needed tool to show that
Theorem \ref{theop} holds. In particular, we construct the announced
sequence converging to the local time $L$.

In the remaining of this paper, $T_i$ denotes the $i$-th jump time of $N$,
with $T_0=0$. It is known that $T_{i}$ has gamma distribution with
parameters $(i,\alpha )$, $i\geq 1.$

We will use the following technical resul in the proofs of this section.

\begin{lemma}
\label{cppax} Let $x\in \mathbb{R}$, $s>0$, $\Omega _{1}(s)=\left\{ \Delta
X_{s}\neq 0\right\}$, and 
\begin{eqnarray*}
\Omega _{2}&=&\{X_{s-}=x,\ \Delta X_{s}\neq 0\ \text{for some }s>0\} \\
&&\cup \{X_{s}=x,\ \Delta X_{s}\neq 0\ \text{for some }s>0\}.
\end{eqnarray*}
Then, $P(\Omega _{1}(s))=0$ and $P(\Omega _{2})=0$.
\end{lemma}

\begin{proof}
By the law of total probability%
\begin{equation*}
P(\Omega _{1}(s))=\sum_{k=0}^{\infty }P(N_{s}=k)P(\Omega _{1}(s)|N_{s}=k).
\end{equation*}%
Notice that%
\begin{equation*}
P(\Omega _{1}(s)|N_{s}=k)=P(\Delta X_{s}\neq 0|N_{s}=k)=P(T_{k}=s)=0.
\end{equation*}

On the other hand, let $\nu \in \mathbb{N}$ and define

\begin{eqnarray*}
\tilde{\Omega}_{\nu }&=&\{X_{s-}=x,\ \Delta X_{s}\neq 0\ \text{for some }%
0<s<\nu \} \\
& &\cup \{X_{s}=x,\ \Delta X_{s}\neq 0\ \text{for some }0<s<\nu \}.
\end{eqnarray*}

For $k=0,$ 
\begin{equation*}
P(\tilde{\Omega}_{\nu }|N_{\nu }=0)=P(\emptyset |N_{\nu }=0)=0,
\end{equation*}%
and for $k\geq 1$,%
\begin{eqnarray*}
P(\tilde{\Omega}_{\nu }|N_{\nu }=k) &\leq &P(X_{T_{j}-}=x\text{ for some }%
j\in \left\{ 1,...,k\right\} |N_{\nu }=k) \\
&&+P(X_{T_{j}}=x\text{ for some }j\in \left\{ 1,...,k\right\} |N_{\nu }=k) \\
&\leq &\sum_{j=1}^{k}(P(X_{T_{j}-}=x|N_{\nu }=k)+P(X_{T_{j}}=x|N_{\nu }=k)).
\end{eqnarray*}%
For $j=1$ we get 
\begin{eqnarray*}
P(X_{T_{1}-} =x|N_{\nu }=k)&=&P(T_{1}=(x-x_{0})c^{-1}|N_{\nu }=k)=0, \\
P(X_{T_{1}} =x|N_{\nu }=k)&=&P(cT_{1}-R_{1}=x-x_{0}|N_{\nu }=k)=0,
\end{eqnarray*}%
this is because $T_{1}$ and $R_{1}$ are independent and absolutely
continuous random variables. Let $P(\cdot |N_{\nu }=k)=P^{\ast }(\cdot )$.
When $j>1$ we have 
\begin{eqnarray*}
P^{\ast }(X_{T_{j}-}=x) &=&\int_{\mathbb{R}}P^{\ast
}(X_{T_{j}-}=x|X_{T_{j-1}-}=y)P^{\ast }(X_{T_{j-1}-}\in dy) \\
&=&\int_{\mathbb{R}}P^{\ast }(R_{j-1}=y-(x-(T_{j}-T_{j-1})c))P^{\ast
}(X_{T_{j-1}-}\in dy) \\
&=&0
\end{eqnarray*}%
and%
\begin{eqnarray*}
P^{\ast }(X_{T_{j}}=x) &=&\int_{\mathbb{R}}P^{\ast
}(X_{T_{j}}=x|X_{T_{j-1}}=y)P^{\ast }(X_{T_{j-1}}\in dy) \\
&=&\int_{\mathbb{R}}P^{\ast }((T_{j}-T_{j-1})c-R_{j}=x-y)P^{\ast
}(X_{T_{j-1}}\in dy) \\
&=&0.
\end{eqnarray*}%
Here we have used the fact that $R_{j-1}$ has an absolutely continuous
distribution. Finally notice that $P(\Omega _{2})\leq \sum {}_{\nu
=1}^{\infty }P(\tilde{\Omega}_{\nu })=0.$
\end{proof}

\subsection{An approximating sequence of the local time}

Now we approximate the local time $L$ by a sequence of suitable random
fields, which allows us to see that Theorem \ref{theop}.a) is true. Toward
this end, let $x\in \mathbb{R}$ arbitrary and fixed. For each $n\in \mathbb{N%
}$\ define\ $\varphi _{x,n}:\mathbb{R}\rightarrow \mathbb{R}$ by 
\begin{equation*}
\varphi _{x,n}(y)=\left\{ 
\begin{tabular}{ll}
$0,$ & $y<x-1/n,$ \\ 
$(n(y-x)+1)/2,$ & $x-1/n\leq y\leq x+1/n,$ \\ 
$1,$ & $x+1/n<y\text{.}$%
\end{tabular}%
\right.
\end{equation*}

Notice that 
\begin{eqnarray}
\lim_{n\rightarrow \infty }\varphi _{x,n}(y) &=&\left\{ 
\begin{tabular}{ll}
$0,$ & $y<x,$ \\ 
$1/2,$ & $y=x,$ \\ 
$1,$ & $y>x,$%
\end{tabular}%
\right.  \notag \\
&=&\frac{1}{2}1_{\{x\}}(y)+1_{(x,+\infty )}(y),  \label{cfn}
\end{eqnarray}%
and 
\begin{equation*}
\varphi _{x,n}^{\prime }(y)=\left\{ 
\begin{tabular}{ll}
$0,$ & $y<x-1/n,$ \\ 
$n/2,$ & $x-1/n<y<x+1/n,$ \\ 
$0,$ & $x+1/n<y.$%
\end{tabular}%
\right.
\end{equation*}%
For each $n\in \mathbb{N}$, we define the random field%
\begin{eqnarray*}
L_{t}^{n}(x) &=&\frac{1}{c}\left( \varphi _{x,n}(X_{t})-\varphi
_{x,n}(X_{0})\right. \\
&&-\sum_{s\leq t}\{\varphi _{x,n}(X_{s})-\varphi _{x,n}(X_{s-})-\varphi
_{x,n}^{\prime }(X_{s-})\Delta X_{s}\}),
\end{eqnarray*}%
where $\Delta X_{t}=X_{t}-X_{t-}$. As in (\ref{tlrp1}), we have by (\ref%
{fsptp}) that $L^{n}$ is well-defined.

Before proving that $\{L^n, n\in\mathbb{N}\}$ is the sequence that we are
looking for, we need to approximate the fuction $\varphi_{x,n}$ by a
sequence of smooth functions. To do so, set 
\begin{eqnarray*}
\Omega ^{\prime } &=&(\left\{ X_{s-}=x\pm 1/m\neq X_{s},\text{ for some }%
s>0,\ m\in \mathbb{N}\right\} \\
&&\cup \{N_{s}<+\infty, \ \text{ for all }\ s>0\}^{c}\cup \Omega _{2})^{c}.
\end{eqnarray*}%
Since, by Lemma \ref{cppax}, 
\begin{eqnarray*}
P(X_{s-} &=&x\pm 1/m\neq X_{s},\ \text{for some }s>0,\ m\in \mathbb{N}) \\
&\leq &\sum_{m=1}^{\infty }P(X_{s-}=x\pm 1/m\neq X_{s},\ \text{for some }%
s>0) =0,
\end{eqnarray*}%
we have $P(\Omega ^{\prime })=1.$

Let $\psi :\mathbb{R}\rightarrow \mathbb{R}$ a symmetric function in $%
\mathcal{C}^{\infty }(\mathbb{R})$ with compact support on $[-1,1]$ and 
\begin{equation*}
\int_{-1}^{1}\psi (y)dy=1.
\end{equation*}%
Define the sequence $\left( \psi _{m}\right) $ by 
\begin{equation*}
\psi _{m}(y)=m\psi (my),\quad y\in \mathbb{R},
\end{equation*}%
and 
\begin{equation*}
\varphi _{x,n}^{m}(y)=(\psi _{m}\ast \varphi _{x,n})(y)=\int_{\mathbb{R}%
}\varphi _{x,n}(y-z)\psi _{m}(z)dz.
\end{equation*}%
Since $\psi _{m}\in \mathcal{C}^{\infty }(\mathbb{R})$, then $\varphi
_{x,n}^{m}\in \mathcal{C}^{\infty }(\mathbb{R})$ and moreover 
\begin{eqnarray}
&&\left( \varphi _{x,n}^{m}\right) _{m}\text{ converges uniformly on
compacts to }\varphi _{x,n},  \label{c1} \\
&&((\varphi _{x,n}^{m})^{\prime })_{m}\text{ converges pointwise, except on }%
x\pm 1/n\text{, to }\left( \varphi _{x,n}\right) ^{\prime }.  \label{c2}
\end{eqnarray}

Let us use the notation 
\begin{equation*}
L_{t}^{n,m}(x)=\frac{1}{c}\int_{(0,t]}(\varphi _{x,n}^{m})^{\prime
}(X_{s-})dX_{s}.
\end{equation*}%
Then, by the change of variable theorem for the Lebesgue-Stieltjes integral,
we have 
\begin{eqnarray}
cL_{t}^{n,m}(x) &=&\varphi _{x,n}^{m}(X_{t})-\varphi _{x,n}^{m}(x_{0}) 
\notag \\
&&-\sum_{0<s\leq t}\left\{ \varphi _{x,n}^{m}(X_{s})-\varphi
_{x,n}^{m}(X_{s-})-(\varphi _{x,n}^{m})^{\prime }(X_{s-})\Delta
X_{s}\right\} .  \label{appfcv}
\end{eqnarray}

Now we can give the relation between $L^{n}$ and $\{L^{n,m},m\in \mathbb{N}\}
$.

\begin{proposition}
\label{capptlapx} Let $n\in \mathbb{N}$. Then, 
\begin{equation*}
\lim_{m\rightarrow \infty }L_{t}^{n,m}(x)=L_{t}^{n}(x),\quad \text{for all}\
t>0,\ \mathit{w.p.1}.
\end{equation*}
\end{proposition}

\begin{proof}
For $\omega \in \Omega ^{\prime }$ we have that (\ref{fsptp}) and (\ref{c1})
imply 
\begin{multline*}
\lim_{m\rightarrow \infty }(\varphi _{x,n}^{m}(X_{t})-\varphi
_{x,n}^{m}(x_{0})-\sum_{0<s\leq t}\left\{ \varphi _{x,n}^{m}(X_{s})-\varphi
_{x,n}^{m}(X_{s-})\right\} ) \\
=\varphi _{x,n}(X_{t})-\varphi _{x,n}(x_{0})-\sum_{0<s\leq t}\left\{ \varphi
_{x,n}(X_{s})-\varphi _{x,n}(X_{s-})\right\} .
\end{multline*}%
Now we analyze the remaining term in the definition of $L_{t}^{n,m}(x).$
Notice that for each $w\in \Omega ^{\prime }$ we have 
\begin{equation*}
X_{s-}(w)\neq x\pm 1/k,\quad k\in \mathbb{N}.
\end{equation*}%
Therefore, from $(\ref{c2}),$%
\begin{equation*}
\lim_{m\rightarrow \infty }\sum_{0<s\leq t}(\varphi _{x,n}^{m})^{\prime
}(X_{s-})\Delta X_{s}=\sum_{0<s\leq t}(\varphi _{x,n})^{\prime
}(X_{s-})\Delta X_{s}.
\end{equation*}%
From this and (\ref{appfcv}) the result follows.
\end{proof}

Now we are ready to state the properties of $\{L^n,n\in\mathbb{N}\}$ that we
use in Section \ref{sec:3}.

\begin{proposition}
\label{aptlft}The sequence $\{L^{n},n\in \mathbb{N}\}$ fulfill:

\begin{itemize}
\item[a)] $L_{t}^{n}(x)=\int_{0}^{t}\varphi _{x,n}^{\prime }(X_{s-})ds+\frac{%
1}{c}\sum_{0<s\leq t}\varphi _{x,n}^{\prime }(X_{s-})\Delta X_{s}$, for all $%
n\in \mathbb{N}$ and $t>0$, w.p.1.

\item[b)] $\lim_{n\rightarrow \infty }L_{t}^{n}(x)=L_{t}(x)$, for all $t>0$,
w.p.1.
\end{itemize}
\end{proposition}

\begin{proof}
We first deal with Statement a). Fix $t\geq 0$ and let $\omega \in \Omega
^{\prime }\cap \Omega _{1}(t)^{c}$. Then there is $k\in \mathbb{N}$ such
that $N_{t}(\omega )=k$. Thus 
\begin{align*}
& \int_{0}^{t}(\varphi _{x,n}^{m})^{\prime }(X_{s-})dX_{s} \\
& =\sum_{i=1}^{k}\int_{(T_{i-1},T_{i}]}(\varphi _{x,n}^{m})^{\prime
}(X_{s-})dX_{s}+\int_{(T_{k},t]}(\varphi _{x,n}^{m})^{\prime }(X_{s-})dX_{s}
\\
& =\sum_{i=1}^{k}\int_{(T_{i-1},T_{i}]}\left( \varphi _{x,n}^{m}\right)
^{\prime }(X_{s-})cds+\int_{[T_{k},t)}\left( \varphi _{x,n}^{m}\right)
^{\prime }(X_{s})cds \\
& \ \ \ +\sum_{i=1}^{k}\left( \varphi _{x,n}^{m}\right) ^{\prime
}(X_{T_{i}-})(X_{T_{i}}-X_{T_{i}-}) \\
& =c\int_{0}^{t}(\varphi _{x,n}^{m})^{\prime
}(X_{s-})ds+\sum_{i=1}^{k}\left( \varphi _{x,n}^{m}\right) ^{\prime
}(X_{T_{i}-})(X_{T_{i}}-X_{T_{i}-}).
\end{align*}%
Notice that on each $(T_{i-1},T_{i}]$ and $(T_{k},t]$, there is at most one $%
s$ such that $X_{s-}=x\pm 1/n$. Hence, by $(\ref{c2})$, we have 
\begin{equation*}
\lim_{m\rightarrow \infty }(\varphi _{x,n}^{m})^{\prime }(X_{s-})=(\varphi
_{x,n})^{\prime }(X_{s-}),\ \ \lambda \text{-}a.s.
\end{equation*}%
Therefore, by the dominated convergence theorem, we deduce 
\begin{eqnarray*}
\lim_{m\rightarrow \infty }\int_{0}^{t}(\varphi _{x,n}^{m})^{\prime
}(X_{s-})dX_{s} &=&c\int_{0}^{t}(\varphi _{x,n})^{\prime }(X_{s-})ds \\
&&+\sum_{i=1}^{k}\left( \varphi _{x,n}\right) ^{\prime }(X_{T_{i}-})\Delta
X_{T_{i}},\ \ a.s.
\end{eqnarray*}%
Consequently, the fact that $L_{t}^{n}(x)$ and the right-hand side of last
equality are c\`{a}dl\`{a}g processes implies that Statement a) holds.

Now we consider Statement b) in order to finish the proof of the
proposition. Let $\omega \in \Omega ^{\prime }$ and $t\geq 0$. Then there
exist $k\in \mathbb{N}$ such that $N_{t}(\omega )=k$. Hence%
\begin{multline*}
\lim_{n\rightarrow \infty }(\varphi _{x,n}(X_{t}(\omega ))-\varphi
_{x,n}(x_{0})-\sum_{i=1}^{k}\left\{ \varphi _{x,n}(X_{T_{i}}(\omega
))-\varphi _{x,n}(X_{T_{i}-}(\omega ))\right\} ) \\
=\frac{1}{2}1_{\left\{ x\right\} }(X_{t}(\omega ))+1_{(x,\infty
)}(X_{t}(\omega ))-\frac{1}{2}1_{\left\{ x\right\} }(x_{0})-1_{\left(
x,\infty \right) }\left( x_{0}\right) \\
-\sum_{i=1}^{k}\{\frac{1}{2}1_{\left\{ x\right\} }(X_{T_{i}}(\omega
))+1_{\left( x,\infty \right) }\left( X_{T_{i}}\left( \omega \right) \right)
\\
-\frac{1}{2}1_{\left\{ x\right\} }\left( X_{T_{i}-}\left( \omega \right)
\right) -1_{\left( x,\infty \right) }\left( X_{T_{i}-}\left( \omega \right)
\right) \}.
\end{multline*}%
On the other hand 
\begin{equation*}
\omega \in \left\{ X_{s-}=x\neq X_{s},\text{ for some }s>0\right\} ^{c}
\end{equation*}%
implies 
\begin{equation*}
X_{T_{i}-}\left( w\right) \neq x,\quad i=0,1,...,k.
\end{equation*}%
Here there exist a finite number of indexes $i$ such that 
\begin{equation*}
X_{T_{i}}\left( w\right) \leq x<X_{T_{i}-}\left( w\right) .
\end{equation*}%
For large enough $n$ we have%
\begin{equation*}
X_{T_{i}-}\left( \omega \right) \notin \left( x-1/n,x+1/n\right) ,\quad
i=0,1,...,k.
\end{equation*}%
Therefore 
\begin{align*}
& \lim_{n\rightarrow \infty }\sum_{0<s\leq t}(\varphi _{x,n})^{\prime
}(X_{s-}(\omega ))\Delta X_{s}(\omega ) \\
& =\lim_{n\rightarrow \infty }\sum_{i=1}^{k}(\varphi _{x,n})^{\prime
}(X_{T_{i}-}\left( \omega \right) )\Delta X_{T_{i}}\left( \omega \right) \\
& =\lim_{n\rightarrow \infty }\sum_{i=1}^{k}\frac{n}{2}1_{\left(
x-1/n,x+1/n\right) }(X_{T_{i}-}\left( \omega \right) )\Delta X_{T_{i}}\left(
\omega \right) =0.
\end{align*}%
Hence the proof is complete.
\end{proof}

\section{Proof of Theorem \protect\ref{theop}}

\label{sec:3} The purpose of this section is to give the proof of Theorem %
\ref{theop}. This proof will be divided into three steps. It is worth
mentioning that the proof of Statement a) gives us a sequence of absolutely
continuous random fields that converges to $L$. Namely, the sequence $%
\{\int_{0}^{t}\varphi _{x,n}^{\prime }\left( X_{s-}\right) ds,n\in \mathbb{N}%
\}.$

\subsection{Proof of part a) of Theorem \protect\ref{theop}}

From part a) and b) of Proposition \ref{aptlft} we have%
\begin{equation*}
L_{t}(x)=\lim_{n\rightarrow \infty }L_{t}^{n}\left( x\right)
=\lim_{n\rightarrow \infty }\int_{0}^{t}\varphi _{x,n}^{\prime }\left(
X_{s-}\right) ds.
\end{equation*}%
which yields that $(L_{\cdot }^{n}\left( x\right) )$ is non-negative and
increasing.

\subsection{Proof of part b) of Theorem \protect\ref{theop}}

Since $X_{s}\leq X_{s-}$ we have 
\begin{eqnarray}
L_{t}\left( x\right) &=&\frac{1}{c}(\frac{1}{2}1_{\left\{ X_{t}\right\}
}\left( x\right) -\frac{1}{2}1_{\left\{ x_{0}\right\} }\left( x\right)
+1_{(-\infty ,X_{t})}\left( x\right) -1_{(-\infty ,x_{0})}\left( x\right) 
\notag \\
&&+\sum_{0<s\leq t}1_{\left( X_{s},X_{s-}\right) }(x)).  \label{tlpcr2}
\end{eqnarray}
Suppose for example, $x_{0}>x$, $X_{t}<x$ and $C_{t}\left( x\right) =n.$ Let 
$c_{1},...,c_{n}$ the crossing times with the level $x$. Then, by
hypothesis, there exist jumping times $s_{1}\in \left( 0,c_{1}\right) $,...,$%
s_{n+1}\in \left( c_{n},t\right)$ such that $x\in(X_{s_i},X_{s_i-})$. Hence 
\begin{eqnarray*}
1_{\left( -\infty ,X_{t}\right) }\left( x\right) -1_{(-\infty
,x_{0}]}(x)+\sum_{0<s\leq t}1_{\left( X_{s},X_{s-}\right) }(x) &=&0-1+\left(
n+1\right) \\
&=&C_{t}(x).
\end{eqnarray*}

\subsection{Proof of part c) of Theorem \protect\ref{theop}}

For each $a,b\in \mathbb{R}$ define%
\begin{eqnarray*}
1_{\langle \langle a,b\rangle \rangle } &=&\left\{ 
\begin{tabular}{ll}
$1_{(a,b]},$ & if $a\leq b,$ \\ 
$-1_{(b,a]},$ & if $b<a,$%
\end{tabular}%
\right. \\
&=&1_{(-\infty ,b]}-1_{(-\infty ,a]}.
\end{eqnarray*}%
From this definition immediately follows that%
\begin{equation}
1_{\langle \langle a,b\rangle \rangle }=1_{(a,c]}-1_{(b,c]},\quad a,b\leq c.
\label{ppnf}
\end{equation}%
Using induction on $n$, we can prove that, for $a_{1},...,a_{n}$ real
numbers, 
\begin{equation}
1_{\langle \langle a_{1},a_{2}\rangle \rangle }+\cdots +1_{\langle \langle
a_{n-1},a_{n}\rangle \rangle }=1_{\langle \langle a_{1},a_{n}\rangle \rangle
}.  \label{spnf}
\end{equation}

On the other hand, for almost all $\omega \in \Omega ^{\prime }$, there
exists $k\in \mathbb{N}\cup \{0\}$ such that $\omega \in \{N_{t}=k\}$.
Therefore, 
\begin{eqnarray*}
\int_{0}^{t}g(X_{s})ds
&=&\sum_{i=1}^{k}\int_{(T_{i-1},T_{i}]}g(X_{s})ds+\int_{(T_{k},t]}g(X_{s})ds
\\
&=&\sum_{i=1}^{k}\int_{(T_{i-1},T_{i}]}g(X_{T_{i-1}}+c(s-T_{i-1}))ds \\
&&+\int_{(T_{k},t]}g(X_{T_{k}}+c(s-T_{k}))ds.
\end{eqnarray*}

Taking $x=X_{T_{i-1}}+(s-T_{i-1})c,$ we can write 
\begin{eqnarray*}
\int_{0}^{t}g(X_{s})ds
&=&\sum_{i=1}^{k}\int_{(X_{T_{i-1}},X_{T_{i-1}}+c(T_{i}-T_{i-1})]}g(x)\frac{%
dx}{c} \\
&&+\int_{(X_{T_{k}},X_{T_{k}}+c(t-T_{k})]}g(x)\frac{dx}{c} \\
&=&\sum_{i=1}^{k}\int_{(X_{T_{i-1}},X_{T_{i}-}]}g(x)\frac{dx}{c}%
+\int_{(X_{T_{k}},X_{t}]}g(x)\frac{dx}{c} \\
&=&\frac{1}{c}\int_{\mathbb{R}}
g(x)\sum_{i=1}^{k}1_{(X_{T_{i-1}},X_{T_{i}-}]}(x)dx \\
&&+\frac{1}{c}\int_{\mathbb{R}} g(x)1_{(X_{T_{k}},X_{t}]}(x)dx.
\end{eqnarray*}%
From (\ref{ppnf}) and (\ref{spnf}) we have%
\begin{eqnarray*}
\int_{0}^{t}g(X_{s})ds &=&\frac{1}{c}\int_{\mathbb{R}}g(x)%
\sum_{i=1}^{k}1_{(X_{T_{i}},X_{T_{i}-}]}(x)dx \\
&&+\frac{1}{c}\int_{\mathbb{R}}g(x)\sum_{i=1}^{k}1_{\langle \langle
X_{T_{i-1}},X_{T_{i}}\rangle \rangle }(x)dx \\
&&+\frac{1}{c}\int_{\mathbb{R}}g(x)1_{(X_{T_{k}},X_{t}]}(x)dx \\
&=&\frac{1}{c}\int_{\mathbb{R}}g(x)%
\sum_{i=1}^{k}1_{(X_{T_{i}},X_{T_{i}-}]}(x)dx \\
&&+\frac{1}{c}\int_{\mathbb{R}}g(x)1_{\langle \langle
X_{T_{0}},X_{T_{k}}\rangle \rangle }(x)dx \\
&&+\frac{1}{c}\int_{\mathbb{R}}g(x)1_{(X_{T_{k}},X_{t}]}(x)dx \\
&=&\frac{1}{c}\int_{\mathbb{R}}(1_{\langle \langle X_{0},X_{t}\rangle
\rangle }+\sum_{i=1}^{k}1_{(X_{T_{i}},X_{T_{i}-}]})(x)g(x)dx \\
&=&\int_{\mathbb{R}}\frac{1}{c}(1_{(-\infty ,X_{t}]}-1_{(-\infty
,X_{0}]}+\sum_{i=1}^{k}1_{(X_{T_{i}},X_{T_{i}-}]})(x)g(x)dx.
\end{eqnarray*}%
Thus, the proof is complete by (\ref{tlpcr2}).

\section{An occupation measure result}

\label{sec:4}

By $F(\cdot ,t)$ we denote the distribution of $%
\sum_{k=1}^{N_{t}}R_{k}1_{[N_{t}>0]}$, and by $f(\cdot ,t)$ the density of $%
F(\cdot ,t),$ when it exists. In order to use Theorem \ref{teomed2} we need
an expression for $E[L_{t}(x)]$, which is given in \cite{T-J-J} (Proposition
1). Namely, if $f\in L^{1}(\mathbb{R}\times \lbrack 0,t])$, then%
\begin{equation}
E[L_{t}(x)]=\int_{[((x-x_{0})/c)\vee 0]\wedge t}^{t}f(x_{0}+cs-x,s)ds.
\label{exprelt}
\end{equation}

\subsection{Example}

Consider the measurable set%
\begin{equation*}
\Delta =[0,\infty )\times \lbrack 0,\infty )\in \mathcal{B}(\mathbb{R}^2).
\end{equation*}%
Then, from Theorem \ref{teomed2} and (\ref{exprelt}), we get%
\begin{align*}
& E[\int_{0}^{t}1_{\Delta }(X_{s},X_{s+\varepsilon }-X_{s})ds] \\
& =\int_{\mathbb{R}}E[1_{\Delta }(x,X_{\varepsilon
}-x_{0})]\int_{[((x-x_{0})/c)\vee 0]\wedge t}^{t}f(x_{0}+cs-x,s)dsdx \\
& =\int_{0}^{\infty }P(x_{0}\leq X_{\varepsilon })\int_{[((x-x_{0})/c)\vee
0]\wedge t}^{t}f(x_{0}+cs-x,s)dsdx \\
& =\int_{0}^{\infty }P(\sum_{k=1}^{N_{\varepsilon }}R_{k}\leq c\varepsilon
)\int_{[((x-x_{0})/c)\vee 0]\wedge t}^{t}f(x_{0}+cs-x,s)dsdx \\
& =\int_{0}^{\infty }F(c\varepsilon ,\varepsilon )\int_{[((x-x_{0})/c)\vee
0]\wedge t}^{t}f(x_{0}+cs-x,s)dsdx.
\end{align*}

Now assume that $R_{1}$ has exponential distribution with parameter $\beta $%
, then the density of $\sum_{k=1}^{N_{t}}R_{k}1_{[N_{t}>0]}$ is 
\begin{equation*}
f(x,t)=e^{-\alpha t-\beta x}\left( \sum_{n=1}^{\infty }\frac{(\beta \alpha
t)^{n}x^{n-1}}{n!(n-1)!}\right) 1_{(0,\infty )}(x),\quad t>0.
\end{equation*}%
Hence, in this case,%
\begin{align*}
& E[\int_{0}^{t}1_{\Delta }(X_{s},X_{s+\varepsilon }-X_{s})ds] \\
& =\int_{0}^{\infty }\left[ \int_{0}^{c\varepsilon }e^{-\alpha \varepsilon
-\beta y}\sum_{n=1}^{\infty }\frac{(\beta \alpha \varepsilon )^{n}y^{n-1}}{%
n!(n-1)!}dy+e^{-\alpha \varepsilon }\right] \\
& \times \int_{\lbrack ((x-x_{0})/c)\vee 0]\wedge t}^{t}e^{-\alpha
s}e^{-\beta (x_{0}+cs-x)}\sum_{k=1}^{\infty }\frac{(\beta \alpha
s)^{k}(x_{0}+cs-x)^{k-1}}{k!(k-1)!}dsdx \\
& =\int_{0}^{x_{0}}\left[ \int_{0}^{c\varepsilon }e^{-\alpha \varepsilon
-\beta y}\sum_{n=1}^{\infty }\frac{(\beta \alpha \varepsilon )^{n}y^{n-1}}{%
n!(n-1)!}dy+e^{-\alpha \varepsilon }\right] \\
& \times \int_{0}^{t}e^{-\alpha s}e^{-\beta (x_{0}+cs-x)}\sum_{k=1}^{\infty }%
\frac{(\beta \alpha s)^{k}(x_{0}+cs-x)^{k-1}}{k!(k-1)!}dsdx \\
& +\int_{x_{0}}^{x_{0}+ct}\left[ \int_{0}^{c\varepsilon }e^{-\alpha
\varepsilon -\beta y}\sum_{n=1}^{\infty }\frac{(\beta \alpha \varepsilon
)^{n}y^{n-1}}{n!(n-1)!}dy+e^{-\alpha \varepsilon }\right] \\
& \times \int_{(x-x_{0})/c}^{t}e^{-\alpha s}e^{-\beta
(x_{0}+cs-x)}\sum_{k=1}^{\infty }\frac{(\beta \alpha s)^{k}(x_{0}+cs-x)^{k-1}%
}{k!(k-1)!}dsdx.
\end{align*}%
For example, under the conditions%
\begin{equation*}
x_{0}=4,\ \alpha =1,\ \beta =1,\ c=1.1,\ t=1,
\end{equation*}%
with $\varepsilon =12$ and considering five iterations on the sums we get%
\begin{align*}
& E[\int_{0}^{1}1_{\Delta }(X_{s},X_{s+12}-X_{s})ds] \\
& \approx \int_{0}^{4}\left[ \int_{0}^{13.2}e^{-12-y}\sum_{n=1}^{5}\frac{%
(12)^{n}y^{n-1}}{n!(n-1)!}dy+e^{-12}\right] \\
& \times \int_{0}^{1}e^{-(2.1)s-4+x}\sum_{k=1}^{5}\frac{%
s^{k}(4+(1.1)s-x)^{k-1}}{k!(k-1)!}dsdx \\
& +\int_{4}^{5.1}\left[ \int_{0}^{13.2}e^{-12-y}\sum_{n=1}^{5}\frac{%
(12)^{n}y^{n-1}}{n!(n-1)!}dy+e^{-12}\right] \\
& \times \int_{(x-4)/(1.1)}^{1}e^{-(2.1)s-4+x}\sum_{k=1}^{5}\frac{%
s^{k}(4+(1.1)s-x)^{k-1}}{k!(k-1)!}dsdx \\
& =7.251\times 10^{-3}.
\end{align*}%
Note that this value may help the insurance company to decide if it invests
part of its wealth in another assets.

\subsection{Proof of Theorem \protect\ref{teomed2}}

We will use the monotone class theorem (see, for example, Ethier and Kurtz 
\cite{E-K}, Theorem 4.2) to show that the result holds. Set%
\begin{equation*}
\mathcal{H}=\{\psi :\mathbb{R}^{2}\longrightarrow \mathbb{R},\ \psi \text{
is measurable, bounded and satisfies }(\ref{cm2})\}.
\end{equation*}%
It is not difficult to see that $\mathcal{H}$ is a real linear space and, by
Theorem \ref{theop}, we have%
\begin{equation*}
\int_{\mathbb{R}}E[L_{t}(x)]dx=E[\int_{\mathbb{R}}L_{t}(x)dx]=E[%
\int_{0}^{t}1_{\mathbb{R}}(X_{s})ds]=t.
\end{equation*}%
It means, $1_{\mathbb{R}^{2}}\in \mathcal{H}$. Moreover $\mathcal{H}$ is
closed under monotone convergence: Let $(\psi _{n})\subset \mathcal{H}$,
such that $0\leq \psi _{n}\uparrow \psi $, $\psi $ bounded, then $\psi $ is
measurable and 
\begin{eqnarray*}
E[\int_{0}^{t}\psi (X_{s},X_{s+\varepsilon }-X_{s})ds] &=&\lim_{n\rightarrow
\infty }E[\int_{0}^{t}\psi _{n}(X_{s},X_{s+\varepsilon }-X_{s})ds] \\
&=&\lim_{n\rightarrow \infty }\int_{\mathbb{R}}E[\psi _{n}(x,X_{\varepsilon
}-x_{0})]E[L_{t}(x)]dx \\
&=&\int_{\mathbb{R}}E[\psi (x,X_{\varepsilon }-x_{0})]E[L_{t}(x)]dx,
\end{eqnarray*}%
which gives that $\psi \in \mathcal{H}$.

Now we use the notation%
\begin{equation*}
\mathcal{K}=\{\psi :\mathbb{R}^{2}\longrightarrow \mathbb{R},\ \psi (\cdot
,\cdot \cdot )=1_{A}(\cdot )1_{B}(\cdot \cdot ),\ \ A,B\in \mathcal{B}(%
\mathbb{R})\}.
\end{equation*}%
Then the family $\mathcal{K}$ is closed under multiplication and $\mathcal{K}%
\subset \mathcal{H}$. In fact, by Theorem \ref{theop} we obtain%
\begin{eqnarray*}
\lefteqn{E[\int_{0}^{t}1_{A}(X_{s})1_{B}(X_{s+\varepsilon }-X_{s})ds]} \\
&=&\int_{0}^{t}E[1_{A}(X_{s})]E[1_{B}(X_{s+\varepsilon }-X_{s})]ds \\
&=&\int_{0}^{t}E[1_{A}(X_{s})]E[1_{B}(\varepsilon
c-\sum_{k=N_{s}+1}^{N_{s+\varepsilon }}R_{k})]ds \\
&=&\int_{0}^{t}E[1_{A}(X_{s})]E[1_{B}(\varepsilon
c-\sum_{k=1}^{N_{\varepsilon }}R_{k})]ds \\
&=&\int_{0}^{t}E[1_{A}(X_{s})]E[1_{B}(X_{\varepsilon }-x_{0})]ds \\
&=&E[1_{B}(X_{\varepsilon }-x_{0})]\int_{0}^{t}E[1_{A}(X_{s})]ds \\
&=&E[1_{B}(X_{\varepsilon }-x_{0})]E[\int_{\mathbb{R}}1_{A}(x)L_{t}(x)dx] \\
&=&\int_{\mathbb{R}}E[1_{A}(x)1_{B}(X_{\varepsilon }-x_{0})]E[L_{t}(x)]dx.
\end{eqnarray*}%
Finally, the Dynkin monotone class theorem yields that the proof is finished.

\bigskip

\noindent \textsc{Acknowledgement.} \textit{The last two authors would like
to thank Cinvestav-IPN and Universidad Aut\'onoma de Aguascalientes for
their hospitality during the realization of this work.}


\begin{thebibliography}{99}
\bibitem{A} S. Asmussen (2000). Ruin Probabilities, World Scientific
Publishing Co., Singapure.

\bibitem{Bertoin} J. Bertoin (1996). L\'{e}vy Processes, Cambridge
University Press.

\bibitem{C-Y} S.N. Chiu, C. Yin (2002). \textit{On occupation times for a
risk process with reserve-dependent premium}, Stochastic Models, \textbf{18}%
(2), 245-255.

\bibitem{Chung} K.L. Chung, R.J. Williams (1990). Introduction to Stochastic
Integration, Birkh\"{a}user, Boston.

\bibitem{E-K} S.N. Ethier, T.G. Kurtz (1986). Markov Processes:
Characterizations and Convergence, John Wiley \& Sons, New York.

\bibitem{F-P} P.J. Fitzsimmons, S.C. Port (1990). \textit{Local times,
occupation times, and the Lebesgue measure of the range of a L\'{e}vy process%
}. Seminar on Stochastic Processes, 1989 (San Diego, CA, 1989), 59--73,
Progr. Probab. \textbf{18}, Birkh\"{a}user, Boston.

\bibitem{T-J-J} E.T. Kolkovska, J.A. L\'{o}pez-Mimbela, J. Villa (2005). 
\textit{Occupation measure and local time of classical risk processes,}
Insurance: Mathematics and Economics, \textbf{37}(3), 573-584.

\bibitem{Grandell} J. Grandell (1991). Aspects of Risk Theory,
Springer-Verlag, New York.

\bibitem{Hawkes} J. Hawkes (1986). \textit{Local times as stationary
processes}, K.D. Elworthy (Ed.), From local times to global geometry, Pitman
Research Notes in Math. Vol. \textbf{150}, Chicago 111-120.

\bibitem{L} P. L\'{e}vy (1948). Processus Stochastiques et Mouvement
Brownien, Gauthier-Villars, Paris.

\bibitem{R-S-S-T} T. Rolski, H. Schmidli, V. Schmidt, J. Teugels (1999).
Stochastic Processes for Insurance and Finance, John Wiley \& Sons, New York.
\end{thebibliography}
\end{document}